\makeatletter

\def\eqalign#1{\,\vcenter{\openup\jot\m@th
  \ialign{\strut\hfil$\displaystyle{##}$&$\displaystyle{{}##}$\hfil
      \crcr#1\crcr}}\,}
\def\eqalignno#1{\displ@y \tabskip\@centering
  \halign to\displaywidth{\hfil$\displaystyle{##}$\tabskip\z@skip
    &$\displaystyle{{}##}$\hfil\tabskip\@centering
    &\llap{$##$}\tabskip\z@skip\crcr
    #1\crcr}}
\def\leqalignno#1{\displ@y \tabskip\@centering
  \halign to\displaywidth{\hfil$\displaystyle{##}$\tabskip\z@skip
    &$\displaystyle{{}##}$\hfil\tabskip\@centering
    &\kern-\displaywidth\rlap{$##$}\tabskip\displaywidth\crcr
    #1\crcr}}

\makeatother

\newdimen\pixel \pixel=.00333333 true in

\documentclass[11pt]{article}
\usepackage{amsmath}
\usepackage{amssymb}
\usepackage{epsfig}
\usepackage{fullpage}
\usepackage{color}
\usepackage{tikz}
\definecolor{light-gray}{gray}{0.3}

\makeatletter
\def\bigpar{\bigbreak\@afterindentfalse\@afterheading\ignorespaces}
\def\medpar{\medbreak\@afterindentfalse\@afterheading\ignorespaces}
\def\smallpar{\smallbreak\@afterindentfalse\@afterheading\ignorespaces}

\newlength{\saveindent}
\newenvironment{proof}%
      {\bigpar{\bf Proof:}\ %  previously \sentsp rather than \
             \setlength{\saveindent}{\parindent} %\setlength{\parindent}{0pt}%
                       \ignorespaces}%
      {\stopproof\ignorespaces\bigbreak \setlength{\parindent}{\saveindent}}

      {\bigpar{\bf Proof:}%
             \setlength{\saveindent}{\parindent} %\setlength{\parindent}{0pt}%
                       \,\ignorespaces}%
      {\ignorespaces\bigbreak \setlength{\parindent}{\saveindent}}

      {\bigpar{\bf Proof:}\ %
             \setlength{\saveindent}{\parindent} %\setlength{\parindent}{0pt}%
                       \ignorespaces}%
      {\ignorespaces\bigbreak \setlength{\parindent}{\saveindent}}

\newenvironment{proofof}[1]%
      {\bigpar{\bf#1:}\ %
             \setlength{\saveindent}{\parindent} %\setlength{\parindent}{0pt}%
                       \ignorespaces}%
      {\stopproof\ignorespaces\bigbreak \setlength{\parindent}{\saveindent}}

      {\smallpar{\bf Remark:}\ %  **** \bigpar, and \bigbreak below ****
                       \ignorespaces}%
      {\stopproof\ignorespaces\medbreak \setlength{\parindent}{\saveindent}}

\newenvironment{remark*}%
      {\smallpar{\bf Remark:}\ %  **** \bigpar, and \bigbreak below ****
                       \ignorespaces}%
      {\ignorespaces\medbreak \setlength{\parindent}{\saveindent}}

      {\smallpar{\bf Remarks:}\ % **** ditto ****
                       \ignorespaces}%
      {\stopproof\ignorespaces\medbreak \setlength{\parindent}{\saveindent}}

\newenvironment{remarks*}%
      {\smallpar{\bf Remarks:}\ % **** ditto ****
                       \ignorespaces}%
      {\ignorespaces\medbreak \setlength{\parindent}{\saveindent}}

      {\smallpar{\bf Remark:}\ %
                       \ignorespaces}%
      {\stopproof\ignorespaces\medbreak \setlength{\parindent}{\saveindent}}

      {\smallpar{\bf Remarks:}\ %
                       \ignorespaces}%
      {\stopproof\ignorespaces\medbreak \setlength{\parindent}{\saveindent}}
\makeatother

\newtheorem{theorem}{Theorem}%[section]    % Remove percent sign
\newtheorem{lemma}[theorem]{Lemma}

\newtheorem{example}{Example}
\def\begex{\begin{example}\parindent=0pt \rm}
\def\endex{\end{example}}
\def\square{\vbox{\hrule height.2pt\hbox{\vrule width.2pt height5pt \kern5pt
                                   \vrule width.2pt} \hrule height.2pt}}
\def\stopproof{\hfill \square \smallskip}

\def \path {{\mathcal P}}

\def \adj {{\,\approx\,}}

\def \U {{ \cal U}}
\def \tp {{\theta_{i,j} \path}}

\def\half{{\textstyle{1\over2}}}

\def\threequarters{{\textstyle{3\over4}}}

\def \r {{\bf R}}

\def \r {{ \cal F}}
\def \la {{\langle}}
\def \ra {{\rangle}}

\def \kt {{ {\hat K } }}

\def\r|{{\Bigr\vert}}
\def\l|{{\Bigl\vert}}

\def \kt {{\tilde K}}

\def\phi {\Phi}

\def \yt {{\tilde Y}}

\def\varepsilon{\mathchar"122 }

\def \chi {{\mathbf 1}}

\def\Gamma {{}}

\def\p {{ \mathbb P}}

\def\Sq{{\cal S}_q}
\def\Sq-{{\cal S}_{q-1}}

\def\iorj {{$i$-or-$j$ }}

\def \rr {{random-to-random insertion }}

\newcommand{\lab}{\label}

\newcommand{\be}{\begin{eqnarray}}
\newcommand{\ee}{\end{eqnarray}}

\begin{document}
\title{Improved bounds for the mixing time of the \rr shuffle}
\author{
{\sc Ben Morris}\thanks{Department of Mathematics,
University of California, Davis.
Email:
{\tt morris@math.ucdavis.edu}.
Research partially supported by
NSF grant CNS-1228828.}
\and
{\sc Chuan Qin}\thanks{Department of Mathematics,
University of California, Davis.
Email:
{\tt cqin@math.ucdavis.edu}.
}  }
\date{}
\maketitle
\begin{abstract}
\noindent
We prove an upper bound of $1.5324 n \log n$ for the mixing time of the \rr shuffle,
improving on the best known upper bound of $2 n \log n$. 
Our proof is based on the analysis of a non-Markovian coupling.
\end{abstract}

\setcounter{page}{1}
%-------------------------------------------------------------------------

%-------------------------------------------------------------------------
\section{Introduction} \lab{intro}
How many shuffles does it take to mix up a deck of cards?
Mathematicians have long been attracted to card shuffling problems.
This is partly because of their natural beauty,
and partly
because they provide a 
testing ground for the more general problem of finding the mixing time of a
Markov chain,
which has applications to 
computer science, statistical physics and optimization.

Let $X_t$ 
be a Markov chain on a finite state space $V$
that converges to the uniform distribution.
For probability measures $\mu$ and $\nu$ on $V$,
define the {\it total variation distance}
$|| \mu - \nu || = \half \sum_{x \in V} |\mu(x) - \nu(x) |$,     
and 
define the $\epsilon$-mixing time
\[
\label{mixingtime}
T_{\rm mix}(\epsilon) = \min \{n:   || \Pr(X_t =  \cdot) - \U ||  \leq \epsilon \mbox{ for all $x \in V$}\} \,,
\]
where $\U$ denotes the uniform distribution on $V$.

The \rr shuffle has the following transition 
rule. At each step choose a card uniformly at random, remove it from the deck and
then re-insert in to a random position. It has long been conjectured that 
the mixing time for the  \rr shuffle on $n$ cards 
exhibits {\it cutoff} at a time 
on the order of $n \log n$. That is, there is a constant $c$ such 
that for any $\epsilon \in (0,1)$, the $\epsilon$-mixing time 
%for the \rr shuffle on $n$ cards 
is asymptotic to 
$cn \log n$. 
%(That is,
%$\lim_{n \to \infty} a_n/b_n = 1$. )
It has further been conjectured (see \cite{D}) that  
the constant $c = \threequarters$.

Uyemura-Reyes \cite{U} proved a lower bound of $\half n \log n$. 
This was improved 
by Subag \cite{S}  
to the conjectured value of $\threequarters n \log n$. 
However, a matching upper bound has not been found. 
Diaconis and Saloff-Coste \cite{DS1} used comparison techniques to prove a
$O(n \log n)$ upper bound.  The constant was improved 
by Uyemura-Reyes \cite{U} and then 
by Saloff-Coste and Zuniga \cite{SZ}, 
who proved upper bounds of $4n\log n$ and $2n \log n$, respectively.
The main contribution of this paper is to improve the constant 
in the upper bound to $1.5324$. 
 We achieve this via a non-Markovian coupling that reduces the 
problem of bounding the mixing time to finding the second largest eigenvalue 
of a certain Markov chain on $9$ states. 
We also use the technique of path coupling (see \cite{BD}).

\section{Main result}
For sequences $a_n$ and $b_n$, we write 
$a_n \sim b_n$ if $\lim_{n \to \infty} {a_n \over b_n} = 1$ and
$a_n \lesssim b_n$ if $\limsup_{n \to \infty} {a_n \over b_n} \leq 1$.
Let $P$ be the transition matrix of the \rr shuffle.
Define
\begin{equation}
d(t) = \max_x || P^t(x, \, \cdot) - \U || \enspace.
\end{equation}
When the number of cards is $n$, we write $d_n(t)$ for the value of $d(t)$, and $T_{\rm mix}^{(n)}(\epsilon)$ for the $\epsilon$-mixing time of the \rr shuffle. Our main result is the following upper bound on $T_{\rm mix}^{(n)}(\epsilon)$. 
\begin{theorem}
\label{maintheorem}
For any $\epsilon \in (0,1)$ we have
$T_{\rm mix}^{(n)} (\epsilon) \lesssim 1.5324 n \log n$. 
\end{theorem}
We think of a permutation $\pi$ in $S_n$ as representing the order of a deck of $n$ cards, 
with $\pi(i) = \mbox{position of card $i$}$.  
Say $x$ and $x'$ are {\it adjacent}, and 
write $x \adj x'$, if $x' = (i,j) x$ for a transposition $(i, j)$. We prove the theorem using a path coupling argument (see \cite{BD}) and the following lemma.
\begin{lemma}
\label{mainlemma}
Suppose $a = 0.6526$. If $n$ is sufficiently large and 
$x$ and $x'$ are adjacent permutations in $S_n$,
then
\[
|| P^t(x, \, \cdot) - P^t(x', \, \cdot) || 
\leq 
e^{-at/n} \hspace{1.3cm} \mbox{for all $t>n\log n$}.
\]
\end{lemma}
The proof of Lemma \ref{mainlemma}, which uses a non-Markovian coupling, is deferred to Section \ref{nonmark}.

\begin{proofof}{Proof of Theorem \ref{maintheorem}}
By convexity of the $l^1$-norm, for any state $y$ we have
\begin{equation}
\label{tvbound}
|| P^t(y, \, \cdot) - \U || \leq 
\max_z || P^t(y, \, \cdot) - P^t(z, \, \cdot) || \enspace .
\end{equation}
Since any permutation in $S_n$ can be written as a product of 
at most $n-1$ transpositions, by the triangle inequality the quantity 
on the righthand side of (\ref{tvbound})
is at most 
\[
(n-1) \max_{x \adj x'} || P^t(x, \, \cdot) - P^t(x', \, \cdot) || \enspace .
\]
Let $a = 0.6526$. By Lemma \ref{mainlemma}, if the number of cards $n$ is sufficiently large we have
\begin{equation}\label{eq:maxx}
\max_x || P^t(x, \, \cdot) - \U || \leq (n - 1) e^{-a t/n} \enspace
\end{equation}
for all $t>n\log n$. Substituting $1.5324n\log n$ for $t$ in (\ref{eq:maxx}), we get
%Let  $a$ be a constant that  satisfies $(1.5324)^{-1}< a < 0.6526$.
%Since $a < 0.6526$, 
 \begin{eqnarray}
d_n( 1.5324  n \log n) 
&\leq&
(n - 1) e^{- 1.5324 a \log n} \\
&=& (n-1)/n^{1.5324  a}  \\
&\to& 0,   \hspace{2cm} n\rightarrow \infty,
\end{eqnarray}
since $1.5324 a > 1$. 

\end{proofof}

\section{Proof of Lemma \ref{mainlemma}}
\label{nonmark}
Recall that 
we think of a permutation $\pi$ in $S_n$ as representing the order of a deck of $n$ cards, 
with $\pi(i) = \mbox{position of card $i$}$.  
Let $M_{i,j}: S_n \to S_n$ 
be the operation on permutations that 
removes the card of label $i$ from the deck and re-inserts it 
\[
\left\{\begin{array}{ll}
\mbox{to the right of the card of label $j$} & \mbox{if $i \neq j$;} \\
\mbox{to the leftmost position} & \mbox{if $i = j$.} \\
\end{array}
\right.
\]
We call such operations {\it shuffles.}
If $\la M_1, \dots, M_k\ra$ is sequence of shuffles, 
we write $x M_1 M_2 \cdots M_l$ for $M_k \circ M_{k-1} \cdots M_1(x)$. 

The transition rule for the \rr shuffle can now 
be stated as follows. If the current state is $x$, choose 
a shuffle $M$ uniformly
at random (that is, choose $a$ and $b$ uniformly at random and let $M = M_{a,b}$) and move to $xM$. 

We call the numbers in $\{1, \dots, n\}$ {\it cards.}
If shuffle $M$ removes card $c$ from the deck and then re-inserts it, we
call $M$ a $c$-move. 

If $\path = \la M_1, M_2, \dots \ra$ is a sequence of shuffles, 
we write $ (\path x)_t$ for the permutation $x M_1 \cdots M_t$. Note that
if $\path$ is a sequence of independent uniform random shuffles, then 
$\{ (\path x)_t : t \geq 0 \}$ is the \rr shuffle started at $x$. 

\subsection{The coupling}
Fix a permutation $x$ and $i, j \in \{1, 2, \dots, n\}$. 
The aim of this subsection is to define a coupling of the \rr shuffle starting from 
$x$ and $(i,j) x$, respectively.

For positive integers $k$ 
we will call a sequence $\la M_1, \dots, M_k \ra$ of 
shuffles a {\it $k$-path}. 
For a $k$-path $\path$, 
define
the
{\it $\path$-queue} (or, simply the {\it queue}) as the following Markov chain $\{Q_t: t = 0, \dots, k\}$ on subsets of cards.
Initially, we have $Q_0 = \emptyset$. 
If the queue at time $t$ is $Q_t$,
and the shuffle at time $t$ is $M_{a,b}$, the next queue $Q_{t+1}$ is 
\[
\left\{\begin{array}{ll}
\{i\} & \mbox{if $a = j$;} \\
\{j\} & \mbox{if $a = i$;} \\
Q \cup \{a\} & \mbox{if $a \notin \{i, j\}$ and $b \in Q_t$.} \\
Q - \{a\} & \mbox{otherwise.} \\
\end{array}
\right.
\]
We call a shuffle an \iorj move if it is an $i$-move or a $j$-move.
For $t \leq k$, we call $t$ a {\it good time} if 
\begin{enumerate}
\item $t$ is an \iorj move;
\item there is a time $t' \in \{t+1, \dots, k\}$ 
such that
\begin{enumerate}
\item  $t'$ is the next \iorj move after $t$;
\item  the queue is a singleton at time $t' - 1$; 
\item the card moved at time $t'$ is different from the card moved at time $t$.
\end{enumerate}
\end{enumerate}
Let $T$ be the {\it last} good time in $\{1, \dots, k\}$, 
with $T = \infty$ if there are no good times, and
let $\theta_{i,j} \path$ be the $k$-path obtained from $\path$
by 
reversing the roles of $i$ and $j$ in each shuffle before time $T$ 
(that is, by replacing shuffle $M_{a,b}$ with $M_{\pi(a), \pi(b)}$, where $\pi$ is a transposition 
of $i$ and $j$). 
Note that $\theta_{i,j} \path$ 
has \iorj moves at the same times as $\path$. Furthermore, since the queue is 
reset at the times of \iorj moves, 
the $\theta_{i,j} \path$-queue 
will have the same values as the
$\path$-queue at all times $t \geq T$.
It follows that
the last good time of $\theta_{i,j} \path$ is the same as the last 
good time of $\path$, and 
hence $\theta_{i,j} ( \theta_{i,j}(\path)) = \path$.
Since $\theta_{i,j}$ is its own inverse, it 
is a bijection and hence if $\path$ is a uniform random $k$-path, then 
so is $\theta_{i,j} \path$. 

Let $x' = (i,j) x$. Let $\path$ be a uniform random $k$-path, and 
for $t$ with $0 \leq t \leq k$, define
\[
x_t = (\path x)_t \;\;\;\;\;\;\;
x'_t = ((\tp) x')_t \enspace.
\]
Let $T$ be the last good time of $\path$.
\begin{lemma}\label{lemma3}
If $T <k$ then $x_k = x'_k$.
\end{lemma}
\begin{proof}
Suppose that $T < k$.
Note that at any time $t < T$, the permutation 
$(\path x)_t$ can be obtained from $(\tp x')_t$ by interchanging 
the cards $i$ and $j$. Let $T'$ be the next \iorj move 
after time $T$. Without loss of generality, there is an $i$-move 
at time $T$ and a $j$-move at time $T'$. We claim that for times 
$t$ with $T \leq t < T'$, the permutation $x'_t$ can be obtained from 
$x_t$ by moving only the cards in $Q_t$, as shown in the diagram below.
(In the diagram, the $m$th $X$ in the top row represents the same 
card as the  $m$th $X$ in the bottom row, and $Q$ represents 
all the cards in $Q_t$, in any order.)
\[\begin{array}{ccccccccccc}
x_t:&X&X&X&X&X&X&Q&X&X&X\\
x'_t:&X&X&X&Q&X&X&X&X&X&X
\end{array}\]
To see this, note that it holds at time $T$, when the queue is the singleton $\{j\}$
(since at this time the $i$'s are placed in the same place), and the 
transition rule for the queue process ensures that if it holds 
at time $t$ then it also holds at time $t+1$. The claim 
thus follows by induction. This means that at time $T' - 1$ the permutations differ only 
in the location of card $j$. That is, they are of the form:
\[\begin{array}{ccccccccccc}
x_{T' - 1}:&X&X&X&X&X&X&j&X&X&X\\
x'_{T' - 1}:&X&X&X&j&X&X&X&X&X&X
\end{array}\]
Thus at time $T'$, when card $j$ is removed and then re-inserted into the deck, the 
two permutations become identical, and they remain identical until time $k$. 
\end{proof}
\begin{lemma}\label{lemma4}
Suppose $a =0.6526$. Then for sufficiently large $n$ and $k>n\log n$, we have $\p(T \geq k) \leq e^{-a k/n}$.
\end{lemma}
\begin{proof}
Consider the Markov chain $Y_t$ defined as follows. 
The state space is $\{0, 1, \dots\} \cup \infty$. The chain starts 
in state $\infty$ and remains there until the first 
\iorj move. From this point on, the value of $Y_t$ is the size of the 
queue, until the first time that
either
\begin{enumerate}
\item card $i$ is moved when the queue is $\{i\}$, or
\item card $j$ is moved when the queue is $\{j\}$.
\end{enumerate}
At this point $Y_t$ moves to state $0$, which is an absorbing state.
Note that $T < k$ exactly when $Y_k = 0$.

For $l = 1, 2, \dots$, define
\[
q(l) = 
\left\{\begin{array}{ll}
{1 \over n} & \mbox{if  $l = 1$,} \\
\\
{3n - 1 \over n^2} & \mbox{if  $l = 2$,} \\
\\
{(l - 1)(n - l + 1) \over n^2} & \mbox{if $l \geq 3$;} \\
\end{array}
\right. 
\]
and define
\[
p(l) = 
\left\{\begin{array}{ll}
{n - 2 \over n^2} & \mbox{if  $l = 1$,} \\
\\
{2n - 6 \over n^2} & \mbox{if  $l = 2$,} \\
\\
{ l(n - l -1) \over n^2}        & \mbox{if $l \geq 3$.} \\
\end{array}
\right. 
\]
The transition rule for $Y_t$ can be described 
as follows. 
If the current state is $0$, the next state is $0$.
If the current state is 
$\infty$ the next state is 
\[
\left\{\begin{array}{ll}
1 & \mbox{with probability ${2 \over n}$;} \\
\\
\infty & \mbox{with probability ${n -2 \over n}$.} \\
\end{array}
\right. 
\]
If the current state is $l \in \{1, 2, \dots\}$,
the next state is
\[
\left\{\begin{array}{ll}
l - 1 & \mbox{with probability $q(l)$;} \\
\\
{l + 1}        & \mbox{with probability $p(l)$;} \\
\\
{1}        & \mbox{with probability ${2 \over n}$, if $l \geq 3$;} \\
\\
l & \mbox{with the remaining probability.} \\ 
\end{array}
\right.  
\]
Let $\yt_t$ be the Markov chain 
on $\{0, 1, \ldots, 7 \} \cup \infty$
obtained from $Y_t$ by replacing transitions to state $8$ with 
transitions to $\infty$. That is, if $K$ and $\kt$ denote the 
transition matrices of $Y_t$ and $\yt_t$, respectively, then
\[
\kt(l, m) = 
\left\{\begin{array}{ll}
K(l, m) & \mbox{if $m \in \{0,1,\ldots,7\}$;} \\
\\
K(7,8)        & \mbox{if $l = 7$ and $m = \infty$.} \\
\end{array}
\right.  
\]

The possible transitions of $Y_t$ and $\yt_t$ are indicated by the graph in Figure \ref{graph}.
We claim that if we start with $\yt_0 = Y_0 = \infty$ then the distribution of $\yt_t$ stochastically 
dominates the distrbution of $Y_t$ for all $t$. 
To see this, note that $Y_t$ changes state with probability less than 
$\half$ at each step, and when it changes state, it either makes a $\pm 1$ move 
or it transitions to $1$. Since for $m \in \{1, 2, \dots\} \cup \infty$, 
the transition probability $K(m, 1)$ is decreasing in $m$, it follows that
$Y_t$ is a monotone chain. The claim follows since 
$\yt_t$ is obtained from $Y_t$ by replacing moves to $8$ with moves 
to the (larger) state of $\infty$.
\begin{figure}%[t]
\centering
\begin{tikzpicture}[xscale = 2, yscale = 2] 
\node at (-4, 0) [circle, draw, very thick] (b0) {$0$};
\node at (-3, 0) [circle, draw, very thick] (b1) {$1$};
\node at (-2, 0) [circle, draw, very thick] (b2) {$2$};
\node at (-1, 0) [circle, draw, very thick ] (b3) {$3$};
\node at (0, 0) [circle, draw, very thick] (b4) {$4$};
\node at (1, 0) [circle, draw, very thick] (b5) {$5$};
\node at (2, 0) [circle, draw, very thick] (b6) {$6$};
\node at (3, 0) [circle, draw, very thick] (b7) {$7$};
\node at (4, 0) [circle, draw, very thick] (b8) {$8$};
\node at (3.5, -1) [circle, draw, very thick] (infinity) {$\infty$};
\node at (4.5, 0) {\huge $\dots$};
\draw[very thick, <-] (b0) -- (b1);
\draw[very thick, <->] (b1) -- (b2);
\draw[very thick, <->] (b2) -- (b3);
\draw[very thick, <->] (b3) -- (b4);
\draw[very thick, <->] (b4) -- (b5);
\draw[very thick, <->] (b5) -- (b6);
\draw[very thick, <->] (b6) -- (b7);
\draw[very thick, <->][color = blue] (b7) -- (b8);
\draw[very thick, ->] (b3) to [bend right = 20] (b1); % node [midway] {$2/n$};
\draw[very thick, ->] (b4) to [bend right = 30] (b1);
\draw[very thick, ->] (b5) to [bend right = 40] (b1);
\draw[very thick, ->] (b6) to [bend right = 50] (b1);
\draw[very thick, ->] (b7) to [bend right = 60] (b1);
\draw[very thick, ->] (b8) to [bend right = 70] (b1);
\draw[very thick, ->] (infinity) to [bend left = 10] (b1);
\draw[very thick, color = red, ->] (b7) to (infinity);
%\draw[very thick][color = blue] (b4) -- (b5);
\end{tikzpicture}
\caption{Graph indicating the possible transitions of $Y_t$ and $\yt_t$.
(The blue edge indicates a possible transition of $Y_t$ and the red edge indicates a possible transition of $\yt_t$.)}\label{graph}
\end{figure}
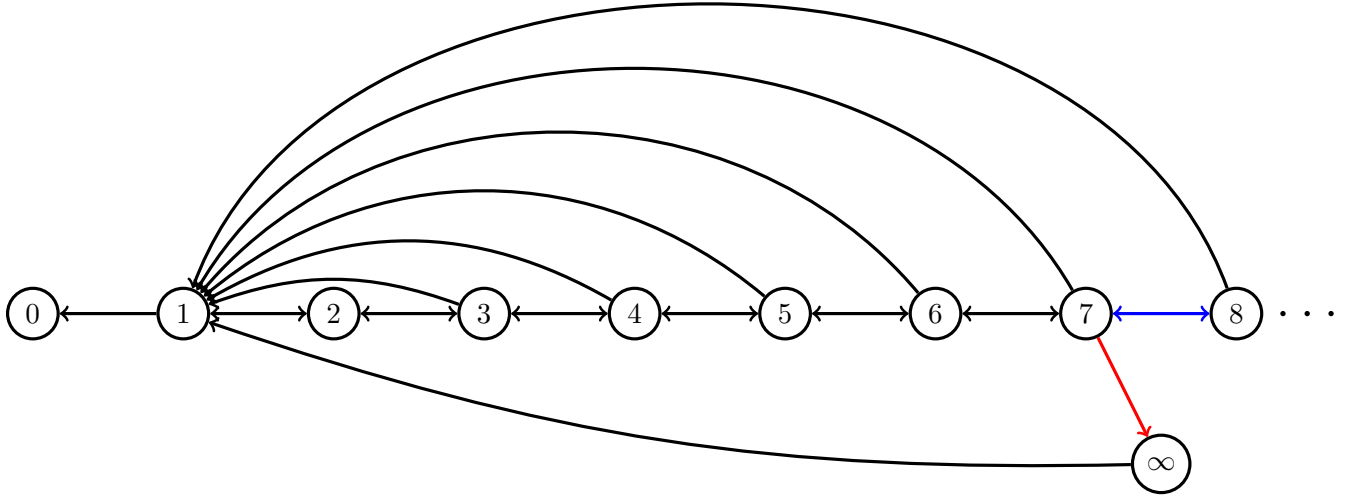

Let $\kt_n$ be the value of the matrix $\kt$ when 
the number of cards is $n$. Define $B_n = \kt_n - I$,
where $I$ is the identity matrix. 
A straightforward calculation shows that 
$n B_n \to C$ as $n\rightarrow \infty$, where
\[
C = \begin{bmatrix} 
0 & 0 & 0 & 0 & 0 & 0  & 0 & 0 & 0 \\ 
1 & -2 & 1 & 0 & 0 & 0  & 0 & 0 &  0  \\ 
0 & 3  & -5 & 2 & 0 & 0 & 0 & 0 & 0    \\ 
0 & 2  & 2 & -7 & 3 & 0  & 0 & 0 & 0  \\ 
0 & 2  & 0 & 3 & -9 & 4  & 0  & 0 & 0  \\ 
0 & 2  & 0 & 0 & 4 & -11 & 5 & 0  & 0  \\ 
0 & 2  & 0 & 0 & 0 &  5  & -13 & 6 & 0\\
0 & 2  & 0 & 0 & 0 &  0  & 6  &-15 & 7\\
0 & 2  & 0 & 0 & 0 &  0 & 0   & 0 & -2
                    \end{bmatrix}
\]
For matrices $A$ we write $\lambda(A)$ for the \emph{second largest} eigenvalue of $A$. By the relationship between $l^2$-norm and eigenvalues, we have
\begin{eqnarray}
\left( \p(\tilde{Y}_k = \infty)^2 +
\sum_{l = 1}^7 \p(\tilde{Y}_k = l)^2  \right)^{1/2}
&\leq& \left(\lambda( \kt_n) \right)^k    \\
&=& [1 + \lambda(B_n)]^k \\
\label{first}
&\leq& \exp\left(  k \lambda(B_n)          \right).
\end{eqnarray}
Since $Y_t$ is stochastically 
dominated by $\yt_t$,
\begin{equation}
\label{third}
\p(Y_k>0)\leq \p(\tilde{Y}_k = \infty) + \sum_{l = 1}^7 \p(\tilde{Y}_k = l) \leq  \sqrt{8} \left( \p(\tilde{Y}_k = \infty)^2 +
\sum_{l = 1}^7 \p(\tilde{Y}_k = l)^2  \right)^{1/2},
\end{equation}
where the second inequality is by Cauchy-Schwarz. Let $J_n$ denote the matrix obtained by deleting the first column and the first row from $I+(1/16)(nB_n)$, and note that $1+\lambda(nB_n)/16$ is the largest eigenvalue of $J_n$. Since $J_n$ is a sub-stochastic matrix for all $n$, it is well known that the largest eigenvalue of $J_n$ converges to the largest eigenvalue of the entry-wise limit of $J_n$ as $n\rightarrow \infty$. This implies that $\lim_{n\rightarrow \infty} \lambda(nB_n) = \lambda(C)$. Numerical calculations show that $\lambda(C) < a:= -0.6526$. Therefore, there exists some constant $\delta>0$ such that
\[
k \lambda(B_n) = {k \over n} \lambda(n B_n)\leq  -(a+\delta) k/n
\]
for sufficiently large $n$. This combined with (\ref{first}) and (\ref{third}) proves that 
\[
\p(T\geq k) = \p(Y_k>0)\leq \sqrt{8} \exp(-(a+\delta) k/n) \leq \exp(-a k/n)
\]
for sufficiently large $n$ and $k>n\log n$.
\end{proof}

\begin{proofof}{Proof of Lemma \ref{mainlemma}}
Recall that for any two probability measures $\mu$ and $\nu$ on a probability space $\Omega$, we have
\[
\|\mu-\nu\|=\min\{\p(X\neq Y): \mbox{$(X,Y)$ is a coupling of $\mu$ and $\nu$} \}.
\]
The main lemma then follows immediately from Lemma \ref{lemma3} and Lemma \ref{lemma4}.
\end{proofof}

%\noindent{\bf{Acknowledgement.}} 

\end{document}